\documentclass[11 pt,reqno]{amsart}
\usepackage{amsmath}
\usepackage{amssymb}
\usepackage{amsthm}
\usepackage{amsfonts}
\usepackage{graphicx}
\usepackage{graphics}
\usepackage{epsfig}

\usepackage{color}

 \newtheorem{thm}{Theorem}[section]
 \newtheorem{cor}[thm]{Corollary}
 \newtheorem{lem}[thm]{Lemma}
 \newtheorem{prop}[thm]{Proposition}
 \theoremstyle{definition}
 \newtheorem{defn}[thm]{Definition}
 
 \newtheorem{prob}[thm]{Problem}
 \theoremstyle{remark}
 \newtheorem{rem}[thm]{Remark}
 \newtheorem{ex}{Example}

\numberwithin{equation}{section} \numberwithin{figure}{section}

\newcommand{\CC}{{\mathbb C}}
\newcommand{\DD}{{\mathbb D}}
\newcommand{\TT}{{\mathbb T}}

\newcommand{\RR}{{\mathbb R}}

\newcommand{\drestr}{\oslash}

\DeclareMathOperator{\Pol}{{\mathcal P}}
\DeclareMathOperator{\dist}{dist}

\begin{document}
\bibliographystyle{alpha}

\title[Cyclicity and extremal polynomials II]{Cyclicity in Dirichlet-type spaces and extremal polynomials II: functions on the bidisk}
\author[B\'en\'eteau]{Catherine B\'en\'eteau}
\address{Department of Mathematics, University of South Florida, 4202 E. Fowler Avenue, Tampa, FL 33620-5700, USA.}
\email{cbenetea@usf.edu}
\author[Condori]{Alberto A. Condori}
\address{Department of Mathematics, Florida Gulf Coast University,
10501 FGCU Boulevard South, Fort Myers, FL 33965-6565, USA.}
\email{acondori@fgcu.edu}
\author[Liaw]{Constanze Liaw}
\address{Department of Mathematics, Baylor University, One Bear Place \#97328, Waco, TX 76798-7328, USA.}
\email{Constanze$\underline{\,\,\,}$Liaw@baylor.edu}
\author[Seco]{Daniel Seco}
\address{Mathematics Institute, Zeeman Building, University of Warwick, Coventry CV4 7AL, UK.} \email{D.Seco@warwick.ac.uk}
\author[Sola]{Alan A. Sola}
\address{Centre for Mathematical Sciences,
University of Cambridge, Wilberforce Road, Cambridge CB3 0WB, UK.}
\email{a.sola@statslab.cam.ac.uk}
\thanks{Liaw is partially supported by the NSF grant DMS-1261687. Seco is supported by ERC Advanced Grant ``Local Structure of Sets, Measures and Currents". Sola acknowledges support from the EPSRC under grant EP/103372X/1.}
\date{\today}

\keywords{Cyclicity, Dirichlet-type spaces, optimal approximation, norm restrictions.}
\subjclass[2010]{Primary: 32A37. Secondary: 32A36, 47A16.}
\begin{abstract}
We study Dirichlet-type spaces $\mathfrak{D}_{\alpha}$ of analytic
functions in the unit bidisk and their cyclic elements. These are
the functions $f$ for which there exists a sequence
$(p_n)_{n=1}^{\infty}$ of polynomials in two variables such that
$\|p_nf-1\|_{\alpha}\to 0$ as $n\to \infty$. We obtain a number of
conditions that imply cyclicity, and 
obtain sharp estimates on
the best possible rate of decay of the norms $\|p_nf-1\|_{\alpha}$,
in terms of the degree of $p_n$, for certain classes of functions using results 
concerning Hilbert spaces of functions of one complex variable and comparisons
between norms in one and two variables.

We give examples of polynomials with no zeros on the bidisk that are not cyclic in $\mathfrak{D}_{\alpha}$ for $\alpha>1/2$ (including the Dirichlet space); this is in contrast with the one-variable case where
all non-vanishing polynomials are cyclic in Dirichlet-type spaces that are not algebras ($\alpha\le 1$). Further, we point out the necessity of a capacity zero condition on zero sets (in an appropriate sense) for cyclicity in the setting of the bidisk, and conclude by stating some open problems.
\end{abstract}

\maketitle

% -------------------------------------------------------------------------

\section{Introduction}
\subsection{Dirichlet-type spaces on the bidisk}
We consider a scale of Hilbert spaces of holomorphic functions on the bidisk
\[\mathbb{D}^2=\{(z_1,z_2)\in \mathbb{C}^2\colon |z_1|<1, |z_2|<1\}\]
indexed by a parameter $\alpha\in (-\infty, \infty)$.
We say that a holomorphic function $f\colon \mathbb{D}^2\to \mathbb{C}$
belongs to the {\it Dirichlet-type space} $\mathfrak{D}_{\alpha}$ if its
power series expansion
\[f(z_1,z_2)=\sum_{k=0}^{\infty}\sum_{l=0}^{\infty}a_{k,l}z_1^kz_2^l\]
satisfies
\begin{equation}
\|f\|^2_{\alpha}=
\sum_{k=0}^{\infty}\sum_{l=0}^{\infty}(k+1)^{\alpha}(l+1)^{\alpha}|a_{k,l}|^2<\infty.
\label{JRnorm}
\end{equation}
Recall that a function of two complex variables is said to be {\it
holomorphic} if it is holomorphic in each variable separately. A
review of the definitions and basic properties such as power series
expansions can be found in \cite[Chapter 2]{HorBook}. Since zero
sets on the boundary of functions $f\in \mathfrak{D}_{\alpha}$ will
play a role later on, we point out
that the topological boundary of the bidisk is much larger than the 
so-called {\it distinguished boundary}
\[\mathbb{T}^2=\{(z_1,z_2)\in
\mathbb{C}^2\colon |z_1|=|z_2|=1\},\]
which is still large enough to support
standard integral representations and the
maximum principle on the bidisk.

The spaces $\mathfrak{D}_{\alpha}$ are a natural generalization to two variables of 
the classical Dirichlet-type spaces $D_{\alpha}$, $-\infty<\alpha<\infty$,
consisting of functions $f(z)=\sum_{k=0}^{\infty}a_kz^k$ that are
analytic in the unit disk $\mathbb{D}=\{z\in \mathbb{C}\colon |z|<1\}$ and
satisfy
\[\|f\|^2_{D_{\alpha}}=\sum_{k=0}^{\infty}(k+1)^{\alpha}|a_k|^2 < \infty;\]
see for instance \cite{Tay66} and \cite{BS84},  
and the references therein. As a remark on notation, we will
continue to use $\| \cdot \|_{\alpha}$ for
the norm of two variable functions in $\mathfrak{D}_{\alpha}$ while $\| \cdot \|_{D_{\alpha}}$
will denote the norm of one variable functions in $D_{\alpha}$. We
point out that the particular choice $\alpha=0$ in $D_{\alpha}$ and
$\mathfrak{D}_{\alpha}$ leads to the classical Hardy spaces $H^2$ on the
disk and bidisk, respectively, while $D_{-1}=A^2(\mathbb{D})$ and
$\mathfrak{D}_{-1}=A^2(\mathbb{D}^2)$ are the canonical Bergman spaces of the
disk and bidisk, and $D_1$ and $\mathfrak{D}_1$ are the Dirichlet spaces of the disk and bidisk, respectively.

The spaces $\mathfrak{D}_{\alpha}$ were studied in detail
by Jupiter and Redett in \cite{JR06}. Spaces of this type appear in the
earlier work of
Kaptano\u{g}lu \cite{Kap94}, which focuses on M\"obius invariance and boundary
behavior in Dirichlet-type spaces, and Hedenmalm \cite{Hed88}, which 
concentrates on closed ideals in function algebras. We note here (cf.
\cite[p. 343]{Kap94} and \cite[Section 4]{Hed88}), that an equivalent
norm for $\mathfrak{D}_{\alpha}$ is given by
\begin{multline*}
\|f\|^2_{\alpha}=|f(0,0)|^2+\int_{\mathbb{D}}|\partial_{z_1}[f(z_1,0)]|^2(1-|z_1|^2)^{1-\alpha}dA(z_1)\\+
\int_{\mathbb{D}}|\partial_{z_2}[f(0,z_2)]|^2(1-|z_2|^2)^{1-\alpha}dA(z_2)\\+
\int_{\mathbb{D}^2}|\partial_{z_2}\partial_{z_1}f(z_1,z_2)|^2(1-|z_1|^2)^{1-\alpha}(1-|z_2|^2)^{1-\alpha}dA(z_1)dA(z_2),
%\label{integralnorm}
\end{multline*}
where $dA(z)=\pi^{-1}dxdy$ denotes area measure. The
proof involves computations with power series, and is omitted.

Extending the earlier one-variable work
of G.D. Taylor in \cite{Tay66}, Jupiter and Redett identified multipliers on
$\mathfrak{D}_{\alpha}$ and studied restriction properties of these spaces. It was also
shown in \cite{JR06} that evaluation at a point in $\mathbb{D}^2$ is a
bounded linear functional,
and hence $\mathfrak{D}_{\alpha}$
%hence that $\mathfrak{D}_{\alpha}$ 
is a {\it reproducing kernel Hilbert space} for all $\alpha$. When $\alpha>1$, the
spaces $\mathfrak{D}_{\alpha}$ are actually {\it algebras} (viz.
the proof of \cite[Theorem 3.1]{JR06}) that are contained (as sets) in
$H^{\infty}(\mathbb{D}^2)$, the algebra of bounded holomorphic functions.
In particular, this implies that for $\alpha>1$, a function $f\in \mathfrak{D}_\alpha$ is cyclic if and only if it has no zeros on the closure of the bidisk.

It is clear from the definition of the norm in \eqref{JRnorm} that any
polynomial $p=p(z_1,z_2)$ belongs to $\mathfrak{D}_{\alpha}$. Moreover, any
$f\in D_{\alpha}$ lifts to $\mathfrak{D}_{\alpha}$ when regarded as constant
%being constant
in one of the variables. In fact, if $g\in D_{\alpha}$ and $h\in D_{\alpha}$,
then the function \[f(z_1,z_2)=g(z_1)h(z_2), \quad (z_1,z_2)\in \mathbb{D}^2,\]
is analytic in the bidisk and belongs to $\mathfrak{D}_{\alpha}$ (see
\cite[Proposition 4.7]{JR06}), and so $\mathfrak{D}_{\alpha}$
certainly contains non-trivial holomorphic functions.

\subsection{Shift operators and cyclic functions}
In this paper, we are interested in a natural pair $\{S_1,S_2\}$ of
bounded linear operators acting on the spaces $\mathfrak{D}_{\alpha}$.
The {\it shift operators} $S_1$ and $S_2$ are defined by setting,
for $f\in \mathfrak{D}_{\alpha}$,
\begin{equation*}
S_1f(z_1,z_2)=z_1f(z_1,z_2) \quad \textrm{and}\quad S_2f(z_1,z_2)=z_2f(z_1,z_2).
%\label{shiftopdef}
\end{equation*}
It is then clear that $S_1$ and $S_2$ are linear, and it follows
from \eqref{JRnorm} that, for every $\alpha$,
$\{S_1,S_2\}$ forms 
%form
a pair of bounded operators mapping $\mathfrak{D}_{\alpha}$
into itself.

It is a standard problem of operator theory to describe the invariant
subspaces of an operator. In the present context, we are interested in
closed subspaces $\mathcal{M}\subset \mathfrak{D}_{\alpha}$ such that
\begin{equation*}
S_1\mathcal{M}\subset \mathcal{M}\quad
\textrm{and}\quad S_2\mathcal{M}\subset \mathcal{M}.
%\label{subspacedef}
\end{equation*}

As a first step towards understanding the invariant subspaces of the pair
$\{S_1,S_2\}$, we seek 
%we would like to find 
conditions under which a function $f\in \mathfrak{D}_{\alpha}$ is
{\it cyclic}, that is, 
%that is, has
\begin{equation*}
[f]=\overline{\textrm{span}\{z_1^kz_2^lf\colon k=0,1, \ldots; l=0, 1, \ldots\}}
=\mathfrak{D}_{\alpha}.
%\label{cyclicity}
\end{equation*}

It is easy to see that there exists at least one cyclic function in each
$\mathfrak{D}_{\alpha}$, namely the function $f(z_1,z_2)=1$. This follows from the
fact that polynomials in two variables are dense in $\mathfrak{D}_{\alpha}$.
%which in turn follows from the fact that the Taylor polynomials
%\[t_n(f)(z_1,z_2)=\sum_{k=0}^n\sum_{l=0}^na_{k,l}z_1^kz_2^l\]
%satisfy
%\[\|f-t_n(f)\|^2_{\alpha}
%=\sum_{k=n+1}^{\infty}\sum_{l=n+1}^{\infty}(k+1)^{\alpha}(l+1)^{\alpha}|a_{k,l}|^2<\epsilon\]
%in view of the norm boundedness condition \eqref{JRnorm},
%provided $n>N(\epsilon)$ is large enough.
On the other hand, since norm convergence implies
uniform convergence on compact subsets, every $g\in [f]$ inherits any zeros $f$ may have inside $\mathbb{D}^2$, and so a necessary condition for cyclicity is that $f(z_1,z_2)\neq 0$, $(z_1,z_2)\in \mathbb{D}^2$. 
% Hence, for a given value of $\alpha$, 
% either all polynomials that vanish on
% $\mathbb{T}$, but not inside the unit disk, are cyclic in
% $D_{\alpha}$, or else they are all non-cyclic.
Note that since $g\in [f]$ implies $[g]\subset [f]$, an equivalent condition for $f$ to
be cyclic in $\mathfrak{D}_{\alpha}$ is that there exists a sequence of
polynomials $(p_n)_{n=1}^{\infty}$ of two variables with
\[\|p_nf-1\|_{\alpha}\to 0,\quad n\to \infty.\]
Since point evaluation is a bounded linear functional, this latter condition
is equivalent to the existence of a sequence of polynomials $(p_n)$ such that
\[p_n(z_1,z_2)f(z_1,z_2)-1\to 0, \quad (z_1,z_2)\in \mathbb{D}^2,\]
and
\[\|p_nf-1\|_{\alpha}\le C.\]
When $\alpha>1$ the spaces $D_{\alpha}$ and $\mathfrak{D}_{\alpha}$ 
are algebras, and cyclic functions have to be non-vanishing on
$\overline{\mathbb{D}}$ and $\overline{\mathbb{D}^2}$, respectively. 

%To see this, note that convergence in norm implies boundedness in norm, and from
%\[|p_{n_k}(z_1,z_2)f(z_1,z_2)-1|\leq C(z_1,z_2) \|p_{n_k}f-1\|_{\alpha}\to 0.\]
%The converse is a consequence of the fact that norm boundedness
%yields existence of a subsequence converging in norm,
%and that the pointwise condition identifies the limit, again
%via boundedness of point evaluations.

In one variable, Beurling characterized the cyclic vectors of
$H^2(\mathbb{D})$: a function $f$ is cyclic if and only if it is
outer. 
In the bidisk, one can show that if $f\in H^2(\mathbb{D}^2)$,
or indeed if $f$ belongs to the Nevanlinna class, then $f$ has
(non-zero) radial limits at almost every $(\zeta_1,
\zeta_2)\in\mathbb{T}^2$.  Thus, 
%, and so 
we can declare $f\in H^2(\mathbb{D}^2)$ to be {\it outer} if
\begin{equation*}
\log|f(z_1,z_2)|=\int_{\mathbb{T}^2}\log|f(e^{i\theta},e^{i\eta})|P((z_1,z_2);(e^{i\theta}, e^{i\eta}))d\theta d\eta;
%\label{outerness}
\end{equation*}
here, $P$ is the product Poisson kernel
\[P((z_1,z_2);(e^{i\theta}, e^{i\eta}))=P_{|z_1|}(\arg{z_1}-\theta)P_{|z_2|}(\arg{z_2}-\eta),\]
where $(z_1,z_2)\in \mathbb{D}^2$ and $\theta, \eta \in [0,2\pi)$.
As usual, $P_r(\theta)=(1-r^2)/(r^2-2r\cos(\theta)+1)^2$ denotes 
%we let $P_r(\theta)=(1-r^2)/(r^2-2r\cos(\theta)+1)^2$ denote
the Poisson kernel of the unit disk.

The cyclicity of $f\in H^2(\mathbb{D}^2)$ does
imply that $f$ is an outer function. But this condition is no longer
sufficient: there are outer functions that are not cyclic (see
\cite[Theorem 4.4.6]{RudBook}); this is another example of how the
higher-dimensional theory is somewhat different. (See however,
\cite{Man88} and \cite{RT10} for some positive results.)
\subsection{Overview of results.}
In the recent paper \cite{BCLSS13}, the problem of cyclicity in Dirichlet-type spaces 
in the unit disk was studied.  More specifically, the authors identified some subclasses of cyclic 
functions and derived sharp estimates on the rate of decay of the norms $\|p_nf-1\|_{\alpha}$ 
for such $f\in D_{\alpha}$.  It seems natural to investigate to what extent these results can 
be extended to functions $f\in\mathfrak{D}_{\alpha}$.

To make the notion of best possible norm decay precise, we let $\mathfrak{P}_n$, $n=1,2,\ldots$ be the subspaces of $\mathfrak{D}_{\alpha}$ consisting of polynomials of two variables of the form
\[p_n=\sum_{k=0}^n\sum_{l=0}^nc_{k,l}z_1^kz_2^l.\]
Note that we regard a monomial $z_1^kz_2^l$ in two variables as having degree $k+l$, meaning that members of $\mathfrak{P}_n$ are polynomials of degree at most $2n$.
Similarly, we denote by $\mathcal{P}_n$ the space of polynomials of one complex variable having degree at most $n$.
We now make the following definition.
\begin{defn}
Let $f \in \mathfrak{D}_\alpha$. We say that a polynomial $p_n \in \mathfrak{P}_n$ is an {\it optimal approximant} of order $n$ to $1/f$ if $p_n$ minimizes
$\|p f-1\|_\alpha$ among all polynomials $p \in \mathfrak{P}_n$.  We call
$\|p_{n}f-1\|_{\alpha}$ the \emph{optimal norm} of order $n$ associated with $f$.
\end{defn}

Stated differently, $p_n$ is an optimal approximant to $1/f$ if we have
\begin{equation*}
\|p_nf-1\|_{\alpha}=\dist_{\mathfrak{D}_{\alpha}}(1, f\cdot \mathfrak{P}_n);
%\label{distdef}
\end{equation*}
here, $\dist_{X}(x,A)=\inf\{\|x-a\|_{X}\colon a\in A\}$ is the usual distance function
between a point and a subset $A\subset X$ of a normed space $X$.

Sharp estimates on the unit disk analog of $\dist_{\mathfrak{D}_{\alpha}}(1,f\cdot \mathfrak{P}_n)$ 
were obtained for certain classes of functions in the paper \cite{BCLSS13}.  To state
these estimates, we define $\varphi_1(s)=\log^+(s)$ for $s\in[0,\infty)$ and, when $\alpha<1$,
\[\varphi_{\alpha}(s)=s^{1-\alpha}, \quad s\in [0,\infty).\]
\begin{thm}[\cite{BCLSS13}, Theorem 3.7]\label{onevariableoptimal}
Let $\alpha\leq 1$.  If $f$ is a function admitting an analytic continuation to the
closed unit disk and whose zeros lie in $\CC\setminus\DD$, then there exists a constant
$C=C(\alpha,f)$ such that
\begin{equation*}
            \dist_{D_{\alpha}}^{2}(1,f\cdot\Pol_{m})
            \leq C\varphi_{\alpha}^{-1}(m+1)
\end{equation*}
holds for all sufficiently large $m$.  Moreover, this estimate is sharp in the sense that if
such a function $f$ has at least one zero on $\TT$, then there exists a constant
$\tilde{C}= \tilde{C}(\alpha,f)$ such that
\begin{equation*}
    \tilde{C} \varphi_{\alpha}^{-1}(m+1) \leq \dist_{D_{\alpha}}^{2}(1,f\cdot\Pol_{m}).
\end{equation*}
\end{thm}

In this paper, we obtain analogous theorems for certain subclasses
of functions in $\mathfrak{D}_{\alpha}$. We begin
Section \ref{productsection} with 
some general remarks concerning cyclicity in
$\mathfrak{D}_{\alpha}$. For instance, if $f$ is
cyclic, then each slice function $f_{z_j}$ obtained when fixing 
the variable $z_j$, $j=1$ or $2$, 
has to be cyclic in $D_{\alpha}$. 
Then the problem of cyclicity and rates associated with optimal approximants 
is addressed
for separable functions, i.e. for functions $f$ of the form
$f(z_1,z_2)=g(z_1)h(z_2)$. We prove that such a function is 
cyclic if and only if the factors $g$ and $h$ are 
%each factor is
cyclic in the one-variable space
$D_{\alpha}$, and then obtain, in Theorem \ref{separable_growth},
sharp estimates on $\dist_{\mathfrak{D}_{\alpha}}(1,f\cdot
\mathfrak{P}_n)$ under the assumption that $g$ and $h$ admit
analytic continuation to the closed disk
%disk, 
and have no zeros in $\mathbb{D}$.

In Section \ref{diagonalsection}, we turn our attention to functions
of the form $f(z_1,z_2)=f(z_1^M\cdot z_2^N)$, for integers $M,N\geq
1$, and again obtain cyclicity results and sharp estimates in
Theorem \ref{sharpcyclicity}. Our proofs are based on the fact that
certain restriction operators furnish isomorphisms between our
subclasses of functions in $\mathfrak{D}_{\alpha}$ and the
one-variable spaces $D_{\alpha}$, and on comparisons between the
associated norms.

In \cite{BCLSS13}, a key role was played by certain Riesz-type 
means of the power series expansion of $1/f$, which turned out to produce
optimal, or near optimal, approximants to $1/f$. The one-variable
construction extends to the bidisk setting as follows. Suppose $1/f$
has formal power series expansion
\[\frac{1}{f(z_1,z_2)}=\sum_{k=0}^{\infty}\sum_{l=0}^{\infty}b_{k,l}z_1^kz_2^l.\]
We then set
\begin{equation}
p_n(z_1,z_2)=\sum_{k=0}^{n}\sum_{l=0}^n\left(1-\frac{\varphi_{\alpha}(\max\{k,l\})}{\varphi_{\alpha}(n+1)}\right)b_{k,l}z_1^kz_2^l.
\label{rieszpolys}
\end{equation}
Note that when $\alpha=0$, the polynomials $p_n$ are simply the $n$th Ces\`aro
means of the Taylor series of $1/f$:
\begin{align*}
C_n(1/f)(z_1,z_2)&=\sum_{k=0}^{n}\sum_{l=0}^{n}\left(1-\frac{\max\{k,l\}}{n+1}\right)b_{k,l}z_1^kz_2^l\\&=\frac{1}{n+1}\sum_{m=0}^nt_m(1/f)(z_1,z_2),
\end{align*}
where $t_m$ denotes the $m$th order Taylor polynomial.  
In Section \ref{examplesSection}, we take a closer look at some concrete polynomials 
in two variables, and show that in some cases the polynomials \eqref{rieszpolys} are indeed close
to optimal.

Recall that in the case of the unit disk, any polynomial that is zero-free 
in $\mathbb{D}$ is cyclic in $D_{\alpha}$ for all $\alpha\leq 1$.  However, the analogous
statement for the bidisk need not hold.  In fact, we give examples of polynomials whose 
zero sets lie in $\mathbb{T}^2$ that are non-cyclic for $\alpha>1/2$, and also polynomials 
with zeros on the boundary of the bidisk that are cyclic for all $\alpha\leq 1$; in
fact, such polynomials can have zero sets that intersect $\mathbb{T}^2$, and extend into 
$\partial \mathbb{D}^2\setminus\mathbb{T}^2$.

The existence of non-cyclic polynomials in Hilbert spaces
of analytic functions in higher dimensions has also been observed by
Richter and Sundberg in setting of the Drury-Arveson space in 
the unit ball of $\mathbb{C}^d$ when $d\geq 4$; 
see \cite{RSslides} for this and other results on cyclic vectors in that 
context.

Many of our results and arguments carry over to the 
$d$-dimensional polydisk $\mathbb{D}^d$, but as notation becomes
much more cumbersome,  we restrict
%have restricted 
our attention to functions on the bidisk.

\section{Classes of cyclic vectors in $\mathfrak{D}_{\alpha}$}\label{productsection}

In this section, we present some examples of cyclic functions in the bidisk.  As a 
preliminary example, we have already observed that $f(z_1,z_2)=1$ is cyclic in 
$\mathfrak{D}_{\alpha}$ for all $\alpha$,
and that cyclic functions cannot vanish inside the bidisk. Moreover, it %It
is not difficult to see that if both $f$ and $1/f$
extend to a larger bidisk, then $f$ is non vanishing on the closure
$\overline{\mathbb{D}}^2$, and $f$ is cyclic; indeed, if $(p_{n})$ 
is a sequence of polynomials such that 
$\|p_n-1/f\|_{\alpha}$ tends to $0$, the estimate 
\[\|p_nf-1\|_{\alpha}\leq\|f\|_{M(\mathfrak{D}_{\alpha})}\|p_n-1/f\|_{\alpha},\]
where $\|\cdot\|_{M(\mathfrak{D}_{\alpha})}$ denotes the multiplier norm,
implies that $1\in[f]$ and so $f$ is cyclic.
%. In that case, we have
%\[\|p_nf-1\|_{\alpha}\leq\|f\|_{M(\mathfrak{D}_{\alpha})}\|p_n-1/f\|_{\alpha},\]
%where $\|\cdot\|_{M(\mathfrak{D}_{\alpha})}$ denotes the multiplier norm, and the 
%right-hand side can be made to tend to $0$ by approximating $1/f$ by polynomials.

However, there do exist cyclic functions in $\mathfrak{D}_{\alpha}$ that vanish on the 
boundary of the bidisk, as in the one variable case.
In this section, we focus on three different ways of building
functions in the bidisk from one variable functions in the unit disk, and explore the 
relationship between the cyclicity in two variables versus that in one variable.
First, let us make some preliminary remarks.

\subsection{Slices of a function}

For a function $f=f(z_1,z_2)$ in the bidisk, we can fix the variable $z_2$, say, and consider the {\it slice}
\[f_{z_2}(z_1)=f(z_1,z_2), \quad z_1\in \mathbb{D},\]
as a function in the unit disk.  The slice $f_{z_1}$ is defined in an analogous manner. With this in mind, the following simple fact holds.

\begin{prop}\label{slice}
If $f$ is cyclic in $\mathfrak{D}_{\alpha},$ then the slices $f_{z_2}$ and $f_{z_1}$ are cyclic in $D_{\alpha}.$
\end{prop}
\begin{proof}
 As a consequence of the Cauchy-Schwarz inequality applied to the coefficients of $f_{z_2}$ we obtain
 \[\|f_{z_2}\|_{D_{\alpha}}\leq\|k_{z_2}\|_{D_{\alpha}}\cdot \|f\|_{\alpha},\]
 where $k_{z_2}$ denotes the reproducing kernel at $z_2$ for $D_\alpha$. Therefore, for any
polynomial $p=p(z_1,z_2)$ we get
 \[\|p_{z_2}f_{z_2}-1\|_{D_{\alpha}}\leq\|k_{z_2}\|_{D_{\alpha}}\cdot \|pf-1\|_{\alpha}.\]
If $f$ is cyclic in $\mathfrak{D}_{\alpha},$ then this last norm
tends to $0$ as the degree of $p$ approaches $\infty$, and therefore for fixed
$z_2$, $\|p_{z_2}f_{z_2}-1\|_{D_{\alpha}}$ approaches $0$ as well.
Consequently, the slice $f_{z_2}$ is cyclic in $D_{\alpha}.$ An
analogous argument applies to the slices in $z_1$, and thus the
result is shown.
\end{proof}
Note that the converse of the above statement does not hold:
consider, for example, $f(z_1,z_2) = 1 - z_1 z_2.$ Then each slice
$f_{z_2}$ and $f_{z_1}$ is non-vanishing in the closed unit disk
(for a fixed $z_2$ and a fixed $z_1$, respectively), and thus each
is cyclic in every $D_{\alpha}$, but it turns out that $f$ is only
cyclic in $\mathfrak{D}_{\alpha}$ for $\alpha \leq 1/2$ (see Remark
\ref{counterex}).

Let us now consider three different natural ways to construct a one variable function from a two variable function and examine issues of cyclicity.
\subsection{Diagonal Restrictions}
The \textit{restriction to the diagonal} of a holomorphic function
on the bidisk produces a function on the disk, and it turns out
that these functions often inherit properties that allow us to
transfer information between one and two variable spaces, see
e.g.~\cite{HoroOber, RudBook}.  For instance, in a recent paper,
Massaneda and Thomas, see \cite{MT13}, were able to use restriction
arguments to show that it is not possible to characterize cyclic
functions in $H^2(\mathbb{D}^2)$ in terms of decay at the boundary.

We define the restriction operator $R_\mathrm{diag}$ on $f\in \mathfrak{D}_{\alpha}$
by
\[R_{\mathrm{diag}}: f\mapsto (\drestr f)(z) %\drestr f(z)
 = f(z,z), \quad z \in \mathbb{D}.\]
To rigorously define which spaces this restriction operator acts on, we define the map
\begin{align*}
%\label{beta}
\beta(\alpha)= \left\{\begin{array}{ll}\alpha-1&\text{for }\alpha\ge 0,\\ 2\alpha-1&\text{for }\alpha< 0.\end{array}\right.
\end{align*}
In order to shorten notation, we use the abbreviation $\beta = \beta(\alpha)$. In the context of the
Dirichlet-type spaces, the following restriction estimate holds.
\begin{prop}
\label{minusone}
If $\alpha\le 2$, then we have
\[
\|\drestr f\|_{D_{\beta}} \le \|f\|_{\alpha} \quad \textrm{for all}\quad  f\in\mathfrak{D}_\alpha.\]
\end{prop}
This result is probably known to the experts, and can be proved by appealing to the theory of reproducing kernels. For the convenience of the reader, we give an
elementary proof.
\begin{proof}[Proof of Proposition \ref{minusone}]
Let $f(z_1,z_2)=\sum_{k=0}^\infty \sum_{l=0}^\infty a_{k,l} z_1^k z_2^l$, which converges absolutely for every $|z_1|< 1$ and $|z_2|< 1.$ Then
$$\drestr f(z) =\sum_{k=0}^\infty \sum_{l=0}^\infty a_{k,l} z^{k+l}$$ converges absolutely for every $|z| < 1$ and can therefore be rewritten as
 $\drestr f (z)= \sum_{n=0}^\infty b_n z^{n},$ where $b_n = \sum_{k+l=n} a_{k,l} = \sum_{k=0}^n a_{k,n-k}.$
Hence,
\[
    \|\drestr f\|_{D_\beta}^2
    = \sum_{n=0}^\infty | b_n|^2(n+1)^\beta
    = \sum_{n=0}^\infty \left|\sum_{k=0}^n a_{k,n-k}\right|^2(n+1)^\beta
\]
and
\begin{align*}
    \|f\|_\alpha^2
    %&=
    %\sum_{k=0}^\infty \sum_{l=0}^\infty |a_{k,l}|^2 (k+1)^\alpha (l+1)^\alpha
    %=
    %\sum_{n=0}^\infty \sum_{k+l=n} |a_{k,l}|^2 (k+1)^\alpha (l+1)^\alpha\\
    &=
    \sum_{n=0}^\infty \sum_{k=0}^n |a_{k,n-k}|^2 (k+1)^\alpha (n-k+1)^\alpha.
\end{align*}

By the Cauchy--Schwarz inequality, we have
\begin{align*}
%\label{K}
&
\left|\sum_{k=0}^n a_{k,n-k}\right|^2 \\
%&
%=
%\left|\sum_{k=0}^n a_{k,n-k}(k+1)^{\alpha/2}(n-k+1)^{\alpha/2}(k+1)^{-\alpha/2}(n-k+1)^{-\alpha/2}\right|^2 \\
&
\le
\left(\sum_{k=0}^n \left|a_{k,n-k}\right|^2(k+1)^{\alpha}(n-k+1)^{\alpha}\right)\left(\sum_{k=0}^n(k+1)^{-\alpha}(n-k+1)^{-\alpha}\right)\\
&\le
\left(\sum_{k=0}^n \left|a_{k,n-k}\right|^2(k+1)^{\alpha}(n-k+1)^{\alpha}\right)
 (n+1)^{-\beta}.
\end{align*}
%where $\beta = \beta(\alpha)$.
In summary, our observations yield
\begin{align*}
&
\|\drestr f\|_{D_\beta}^2
=
 \sum_{n=0}^\infty \left|\sum_{k=0}^n a_{k,n-k}\right|^2(n+1)^\beta\\
&
\le
\sum_{n=0}^\infty \sum_{k=0}^n |a_{k,n-k}|^2 (k+1)^\alpha (n-k+1)^\alpha
=
\|f\|_\alpha^2
\end{align*}
and the proposition is proved.
\end{proof}

This result implies that a function $g \in D_{\beta}$ that arises as
the restriction to the diagonal of a cyclic function in
$\mathfrak{D}_{\alpha}$ is itself cyclic. Viewed differently, a
function of two variables cannot be cyclic in
$\mathfrak{D}_{\alpha}$ unless its restriction $\drestr f$ is cyclic
in $D_{\beta}$ (though it can happen that $\drestr f$ is cyclic, and
$f\in \mathfrak{D}_{\alpha}$ is not); see \cite{MT13} for a
discussion in the context of $H^2(\mathbb{D}^2)$. Moreover, together
with Theorem \ref{onevariableoptimal}, Proposition \ref{minusone}
immediately implies a lower bound for the decay rate of $\|p_n f -
1\|^2_\alpha$ for certain ``nice" functions $f$:
\begin{cor}
Let $\alpha\le 2$.
Suppose $f\in \mathfrak{D}_\alpha$ is such that the diagonal restriction 
$\drestr f$
%$\drestr f =R_{\mathrm{diag}}f$ 
satisfies the hypotheses of Theorem \ref{onevariableoptimal}. Then
\[\|p_n f - 1\|^2_\alpha \ge C \varphi_{\beta}^{-1}(n+1), \quad \textrm{for all}\quad p_n\in \mathfrak{P}_n.\]
\end{cor}

We will see later (see Theorems \ref{characterize}  and \ref{sharpcyclicity}) that this decay rate is not
optimal in general.
Note that the diagonal restrictions of the functions $f(z_1,z_2)=1-z_1 z_2$, $f(z_1,z_2)=(1-z_1)(1-z_2)$,
and $f(z_1,z_2)=1-z_1$  
all satisfy the hypotheses.% of Theorem \ref{onevariableoptimal}.

The above remarks show how, given a cyclic function of two
variables, one can easily obtain examples of cyclic functions of one
variable (although we might need to change the index $\alpha$ of the
space in which cyclicity is being considered!)  
In the next two subsections we examine
how to obtain some classes of cyclic functions
of two variables from cyclic functions of one variable, and we
obtain {\it sharp} rates of decay in some cases.

\subsection{Separable functions}
Let us now consider functions of two variables that can be written as products of two functions of one variable:
\begin{align}
\label{product}
f(z_1,z_2)=g(z_1) h(z_2).
\end{align}
We shall refer to such functions as {\it separable}.
Note that for such products, it follows from \eqref{JRnorm} that $\|f\|_\alpha = \|g\|_{D_\alpha} \|h\|_{D_\alpha}.$

\begin{prop}\label{characterize}
Let $\alpha \in \RR$ and $f$ be defined as in \eqref{product}, where $g,h \in D_{\alpha}$.
Then $f$ is cyclic in $\mathfrak{D}_{\alpha}$ if and only if $g$ and
$h$ are cyclic in $D_{\alpha}$.
\end{prop}
\begin{proof}
First notice that by Proposition \ref{slice}, if $f$ is cyclic in
$\mathfrak{D}_{\alpha}$, then $g$ and $h$ are constant multiples
(with respect to the fixed variable) of the slices of $f$, and thus
are cyclic in $D_{\alpha}$.

For the converse,
%consider the sequence $\{p_n \in \mathcal{P}_n: p_n= p_n(z_1)\}$ that minimizes $\|p_n g - 1\|%%{D_\alpha}$
suppose both $g$ and $h$ are cyclic in $D_{\alpha}$. Let $(p_n)$ and $(q_n)$ be sequences of polynomials such that
$\|p_ng-1\|_{D_{\alpha}}\to 0$ and $\|q_nh-1\|_{D_{\alpha}}\to 0$, respectively.
Since the expression $p_n g h - h = (p_n(z_1) g(z_1) -1 ) h(z_2)$
is separable, we obtain
\begin{align*}
%\label{equal}
\|p_n f- h\|_{\alpha} = \|p_n g - 1\|_{D_\alpha}\|h\|_{D_\alpha}.
\end{align*}
Hence, we get that $h\in [f]$, where $[\cdot]$ denotes the cyclicity class in $\mathfrak{D}_\alpha$, %By \cite[Proposition 5, part 2]{BS84} we get
and so $[h]\subset [f]$. Since $\|q_n h- 1\|_{\alpha} = \|q_n h - 1\|_{D_\alpha}$,
%for all sequences $\{q_n \in \mathcal{P_n}:q_n=q_n(z_2)\}$. Therefore,
the function $h$ is cyclic in $\mathfrak{D}_\alpha$ and $D_\alpha$ simultaneously, and the assertion follows.
\end{proof}

It seems natural to ask whether the growth of the extremal polynomials 
for separable functions is the same as for functions in the unit disk.
As we will see in Theorem \ref{separable_growth}, this is
indeed the case.  Let us first prove a lemma that will help to
establish the sharp growth restrictions.

\begin{lem}\label{optimality}
    Suppose $f=g\cdot h \in \mathfrak{D}_{\alpha}$ for $g,h\in \mathfrak{D}_{\alpha}$, and suppose that $g$ admits a non-vanishing analytic continuation to the closed bidisk.  Then there exists a constant $C$, independent of $n$, such that
        $$ \dist_{\mathfrak{D}_\alpha}(1, f \cdot \mathfrak{P}_n) \geq C \dist_{\mathfrak{D}_\alpha}(1, h \cdot \mathfrak{P}_{2n}).$$
\end{lem}

\begin{proof}
Notice first that since the power series for $g$ converges in a larger polydisk than the unit bidisk, there exists $R > 1$ such that
if $g_n$ are the Taylor polynomials of degree $n$ approximating $g$, the multiplier norm $\|g - g_n\|_{M(\mathfrak{D}_{\alpha})}$ decays exponentially like $R^{-(n+1)}.$
Moreover, since in addition $g$ has no zeros in the closed disk, the multiplier norm $\|1/g\|_{M(\mathfrak{D}_{\alpha})}$ is bounded.

Now let $p_n(z_1,z_2)$ be the optimal approximant to $1/f$ of degree $n$. Then by the above remarks, we have
$$ \|p_nh - 1/g\|_{\alpha} \leq  \|1/g\|_{M(\mathfrak{D}_{\alpha})} \|p_n f - 1\|_{\alpha},$$
which goes to $0$ as $n \rightarrow \infty,$ and therefore in particular, the norms $\|p_nh\|_{\alpha}$ are bounded by some constant $C_1$.
Moreover,
\begin{align*}
\|p_nf - 1 \|_{\alpha} &  =  \|p_n h (g - g_n) + g_n p_n h - 1 \|_{\alpha} \\
& \geq  \|g_n p_n h - 1 \|_{\alpha} - \|p_nh\|_{\alpha} \|g - g_n\|_{M(\mathfrak{D}_{\alpha})}.
\end{align*}
Since $\|p_nh\|_{\alpha}$ is bounded and $\|g - g_n\|_{M(\mathfrak{D}_{\alpha})}$ decays exponentially, we obtain that there exists a constant $C$ such that
$$ \|p_nf - 1 \|_{\alpha}^2 \geq C \dist_{\mathfrak{D}_{\alpha}}(1, h \cdot \mathfrak{P}_{2n}),$$
as desired.
\end{proof}

Using Lemma \ref{optimality}, we obtain sharp estimates on the decay of norms.
\begin{thm}\label{separable_growth}
    Let $\alpha \leq 1$ and $g,h \in D_{\alpha}$. Suppose $g$ and $h$ admit analytic continuations to $\overline{\DD}$ and have no zeros in $\DD.$ Define $f(z_1,z_2) = g(z_1)h(z_2).$ Then
    there exists a constant $C = C(g,h,\alpha)$ such that
        $$  \dist^2_{\mathfrak{D}_\alpha}(1, f \cdot \mathfrak{P}_n) \leq C \varphi^{-1}_{\alpha} (n+1),$$ for all sufficiently large $n$. Moreover, this estimate is sharp in the sense that
    if $h$ has at least one zero on $\TT$ and $g$ has no zeros in the closed disk $\DD$ (or vice versa), then there exists a constant $\tilde{C} = \tilde{C}(g, h, \alpha)$ such that
        $$ \tilde{C} \varphi^{-1}_{\alpha} (n+1) \leq \dist^2_{\mathfrak{D}_\alpha}(1, f \cdot \mathfrak{P}_n).$$
\end{thm}

\begin{proof}
    By Theorem \ref{onevariableoptimal}, for any polynomials $p_n(z_1)$ and $q_n(z_2)$ of degree less than or equal to $n,$ there exist constants $C_1$ and $C_2$ such that
        $$ \| p_n(z_1) g(z_1) - 1 \|_{D_{\alpha}} \leq C_1 \varphi^{-1/2}_{\alpha} (n+1) $$ and
        $$ \| q_n(z_2) h(z_2) - 1 \|_{D_{\alpha}} \leq C_2 \varphi^{-1/2}_{\alpha} (n+1). $$
Therefore
    \begin{align*}
        \| p_n(z_1) q_n(z_2) g(z_1)h(z_2) - 1 \|_{\alpha} & \leq
            \| q_n(z_2) h(z_2) (p_n(z_1)g(z_1) - 1)  \|_{\alpha}\\
                &+ \| q_n(z_2) h(z_2) - 1 \|_{\alpha} \\
        & \leq \|q_n h \|_{\alpha} \| p_n g - 1 \|_{\alpha}
            + \| q_n h - 1 \|_{\alpha} \\
        & =  \|q_n h \|_{D_{\alpha}}  \| p_n g - 1 \|_{D_{\alpha}}
            + \| q_n h - 1 \|_{D_{\alpha}} \\
        & \leq \left( \|q_n h - 1 \|_{D_{\alpha}} + 1 \right)  \| p_n g - 1 \|_{D_{\alpha}}\\
            &+ \| q_n h - 1 \|_{D_{\alpha}} \\
        & \leq C_2 C_1 \varphi^{-1}_{\alpha} (n+1) + (C_1+C_2) \varphi^{-1/2}_{\alpha} (n+1)\\
        & \leq C \varphi^{-1/2}_{\alpha} (n+1)
    \end{align*}
    for some constant $C.$ Therefore,
        $$  \dist^2_{\mathfrak{D}_\alpha}(1, f \cdot \mathfrak{P}_n) \leq C \varphi^{-1}_{\alpha} (n+1),$$ for all sufficiently large $n$, as desired.

Moreover, the inequality is sharp.  To see this,
%, for 
suppose $h$ has at least one zero on $\TT$ and $g$ has no zeros in the closed unit disk.  
Then by Lemma \ref{optimality}, there exists a constant $C_1$ such that
\begin{equation}\label{ineq1}
\dist_{\mathfrak{D}_\alpha}(1, f \cdot \mathfrak{P}_n) \geq C_1 \dist_{\mathfrak{D}_\alpha}(1, h \cdot \mathfrak{P}_{2n}).
\end{equation}
Note that $h=h(z_2)$, and so, by orthogonality of monomials in $\mathfrak{D}_{\alpha}$,  the quantity $\dist_{\mathfrak{D}_\alpha}(1, h \cdot \mathfrak{P}_{2n})$ is bounded from below by $\dist_{\mathfrak{D}_{\alpha}}(1,h\cdot \mathcal{P}_{2n})=
\dist_{D_{\alpha}}(1,h\cdot \mathcal{P}_{2n})$.
Now by Theorem \ref{onevariableoptimal} applied to $h$, and since $\varphi_{\alpha}(2n+1)$ is comparable
to $\varphi_{\alpha}(n+1)$,
there exists a constant $C_2$ such that
\begin{equation}\label{ineq2}
\dist_{{D}_\alpha}^2(1, h \cdot \mathcal{P}_n) \geq C_2 \varphi_{\alpha}^{-1}(n+1).
\end{equation}
Thus, the inequalities in \eqref{ineq1} and \eqref{ineq2} imply
%Putting the two inequalities together gives 
the desired result.
\end{proof}

\section{Norm comparisons and sharp decay of norms for the subspaces $\mathcal{J}_{\alpha, M, N}$}
\label{diagonalsection}

Let us now consider a third way of relating two variable cyclic functions to one variable cyclic functions.
In particular, we shall show that
the polynomials in equation \eqref{rieszpolys} furnish optimal approximants for a certain subclass of functions.
\subsection{The subspaces $\mathcal{J}_{\alpha, M, N}$}
In order to formulate our results, we need some notation.
For $-\infty<\alpha<\infty$ and integers $M,N\geq 1$, we
consider the closed subspaces
%\footnote{Proof for $p=q=1$. Since
%$\mathfrak{D}_{\alpha}$ is complete, Cauchy sequences in $I$ or
%$J$ will converge to some
%$f=\sum_{k,l}a_{k,l}z_1^kz_2^l\in \mathfrak{D}_{\alpha}$. But then, by
%orthogonality, $\|f_n-f\|^2_{\alpha}
%=\sum_{k\neq l}(k+1)^{\alpha}(l+1)^{\alpha}|a_{k,l}|^2+\sum_{k=l}(\cdots)$.
%Hence the first sum is independent of $n$ and must be identically $0$, which
%forces the off-diagonal coefficients of $f$ to be $0$, and so the
%subspaces are closed.}
\begin{equation*}
\mathcal{J}_{\alpha, M, N}=\left\{f\in \mathfrak{D}_{\alpha}\colon f=\sum_{k=0}^{\infty}a_kz_1^{Mk}z_2^{Nk}\right\}.
\end{equation*}
For instance, $\mathcal{J}_{\alpha}=\mathcal{J}_{\alpha,1,1}$ consists
of the functions $f$ whose Taylor coefficients $(a_{k,l})$ vanish off
the diagonal $k=l$, meaning that $f(z_1,z_2)=f(z_1\cdot z_2)$.
The subspace $\mathcal{I}_{\alpha}$
consists of functions that do not depend on $z_2$.
\begin{thm}\label{sharpcyclicity}
Let $f\in \mathcal{J}_{\alpha, M, N}$ have the property that $R(f)=f(z^{1/M},1)$ is
a function that admits an analytic continuation to the closed unit disk, whose
zeros lie in $\mathbb{C}\setminus \mathbb{D}$.

Then
$f$ is cyclic in $\mathfrak{D}_{\alpha}$, and there exists a constant $C=C(\alpha,f,M,N)$ such that
\[\dist_{\mathfrak{D}_{\alpha}}(1,f\cdot \mathfrak{P}_{n})\leq \varphi^{-1}_{2\alpha}(n+1).\]

This result is sharp in the sense that, if $R(f)$ has at least one zero on $\mathbb{T}$, then there exists a constant $c=c(\alpha,f,M,N)$ such that, for large $n$,
\[c\varphi^{-1}_{2\alpha}(n+1)\leq \dist_{\mathfrak{D}_{\alpha}}(1,f\cdot \mathfrak{P}_n).\]

The same conclusions remain valid for $f\in \mathcal{I}_{\alpha}$, 
with the rate $\varphi^{-1}_{2\alpha}$ replaced by $\varphi^{-1}_{\alpha}$.
\end{thm}

We should point out that the hypotheses of Theorem \ref{sharpcyclicity} imply
that $f$ is 
%the Theorem automatically impose the condition that $f$ be
non-vanishing in $\mathbb{D}^2$. For instance, suppose $f\in \mathcal{J}_{\alpha}$ has $f(z_1,z_2)=0$
for some $(z_1,z_2)\in \mathbb{D}^2$. Then the function $R(f)$ will have a zero at
$z=|z_1z_2|e^{i(\arg{z_1}+\arg{z_2})} \in \mathbb{D}$.
\begin{rem}\label{counterex}
It is straight-forward to check that functions like $f(z_1,z_2) = 1 - z_1$, $f(z_1,z_2) = (1 - z_1z_2)^N$, $N\in \mathbb{N}$, and
$f(z_1,z_2)=z_1^2z_2^2-2\cos \theta z_1z_2+1$, $\theta\in \mathbb{R}$, satisfy the assumptions of Theorem \ref{sharpcyclicity}.

The arguments used in the proof of Theorem \ref{sharpcyclicity} imply a function
$f\in \mathcal{J}_{\alpha,M,N}$ can {\it fail} to be cyclic in $\mathfrak{D}_{\alpha}$ 
when $\alpha>1/2$. For instance, the function $f(z_1,z_2)=1-z_1z_2$ is cyclic 
if and only if $\alpha\leq1/2$ (see Example \ref{twistedcircle} below), and the Riesz polynomials 
\eqref{rieszpolys} are optimal approximants to $1/f$ when $\alpha\leq 1/2$.
\end{rem}

\subsection{Liftings, restrictions, and norm comparisons} 
The proof of Theorem \ref{sharpcyclicity} ultimately relies on Theorem 
\ref{onevariableoptimal}, and comparison between the norm of $\mathfrak{D}_{\alpha}$ and 
that of $D_{2\alpha}$.

Suppose that for some real $\alpha$, the function $F=\sum_{k=0}^{\infty}a_kz^k$ belongs to $D_{\alpha}$, 
a Dirichlet-type
space on the unit disk. We define $E\colon D_{\alpha}\to \mathfrak{D}_{\alpha}$ by
\[ E(F)(z_1,z_2)=F(z_1).\]
In addition,
if $f\in \mathcal{I}_{\alpha}$, the mapping $C\colon \mathfrak{D}_{\alpha}\to D_{\alpha}$ given 
by $C(f)(z)=f(z,1)$ is well-defined, and we have $E\circ
C_{|\mathcal{I}_{\alpha}}=\mathrm{id}_{\mathcal{I}_{\alpha}}$. Moreover,
it is immediate that
\begin{equation*}
\|E(F)\|_{\alpha}=\|F\|_{D_{\alpha}},
\quad F \in D_{\alpha}
\end{equation*}
and
\begin{equation*}
\|f\|_{\alpha}=\|C(f)\|_{D_{\alpha}},\quad f\in \mathcal{I}_{\alpha}.
%\label{1varrestrictionnorm}
\end{equation*}

Another embedding is the following one. For
$\alpha \in \mathbb{R}$ fixed, define the mappings
\begin{equation*}
L_{M,N}\colon D_{2\alpha}\to \mathfrak{D}_{\alpha} \quad
\textrm{via}\quad L_{M,N}(F)(z_1,z_2)=F(z_1^M\cdot z_2^N),
\label{liftingop}
\end{equation*}
and
\begin{equation*}
R_{M,N}\colon \mathcal{J}_{\alpha, M, N}\to D_{2\alpha}\quad
\textrm{via}\quad
R_{M,N}(f)(z)=f(z^{1/M},1).
\end{equation*}
We initially view $f(z^{1/M},1)$ as a formal expression,
but the assumption $\sum_{k}(k+1)^{2\alpha}|a_k|^2<\infty$ implies that
$f(z_1^{1/M},1)$ is actually a well-defined holomorphic function on
$\mathbb{D}$; this will become apparent below.
By definition, we again have
$L\circ R_{|\mathcal{J}_{\alpha,M,N}}
=\mathrm{id}_{\mathcal{J}_{\alpha,M,N}}$.
\begin{lem}\label{comparisonlemma}
For $F\in D_{2\alpha}$ and $f\in \mathcal{J}_{\alpha,M,N}$, there are
constants $c_1=c_1(\alpha,M,N)$ and $c_2=c_2(\alpha,M,N)$ such that
\begin{equation*}
\|L_{M,N}(F)\|_{\alpha}\leq c_1\|F\|_{D_{2\alpha}}
%\label{liftingnorm}
\end{equation*}
and
\begin{equation*}
c_2
\|R(f)\|_{D_{2\alpha}}\leq \|f\|_{\alpha}%.
%\label{1restrictionnorm}
\end{equation*}
hold.
In particular, if $f\in \mathcal{J}_{\alpha,M,N}$, then
\begin{equation}
c_2\|R(f)\|_{D_{2\alpha}}\leq \|f\|_{\alpha}\leq c_1\|R(f)\|_{D_{2\alpha}} .
\label{normcomparison}
\end{equation}
\end{lem}
\begin{proof}
We provide the proof of the second inequality; the proof of the first is
analogous.

We first observe that for any $\alpha \in \mathbb{R}$ and $M\geq 1$,
there exist constants $c_1(\alpha,M)$ and $c_2(\alpha,M)$ such that
\[c_1(\alpha,M)(k+1)^{\alpha}\leq (Mk+1)^{\alpha}\leq c_2(\alpha,M)(k+1)^{\alpha},\]
for any $k\in \mathbb{N}$.

Thus, writing $R(f)(z)=\sum_{k=0}^{\infty}a_kz^k$, we have
\begin{align*}
\|R(f)\|^2_{D_{2\alpha}}&=\sum_{k=0}^{\infty}(k+1)^{2\alpha}|a_k|^2\\
&=\sum_{k=0}^{\infty}(k+1)^{\alpha}(k+1)^{\alpha}|a_k|^2\\
&\leq [c_1(\alpha, M)c_1(\alpha,N)]^{-1}
\sum_{k=0}^{\infty}(Mk+1)^{\alpha}(Nk+1)^{\alpha}|a_k|^2\\
&=[c_1(\alpha, M)c_1(\alpha,N)]^{-1}
\|f\|^2_{\alpha},
\end{align*}
which proves the assertion.

The two-sided bound \eqref{normcomparison} follows from the one-sided bounds 
and the fact that $f=L(R(f))$.
\end{proof}
In particular, we see from the proof of Lemma \ref{comparisonlemma}
that in the case $M=N=1$, the equalities
\[\|L(F)\|_{\alpha}=\|F\|_{D_{2\alpha}}\quad \textrm{and}\quad
\|R(f)\|_{D_{2\alpha}}=\|f\|_{\alpha}\]
hold and hence $R$ is an isometric isomorphism between $\mathcal{J}_{\alpha}$ and $D_{2\alpha}$.

\subsection{Sharpness of norm decay}
We shall use Lemma \ref{comparisonlemma}, along with the following lemma,
to prove Theorem \ref{sharpcyclicity}.
\begin{lem}\label{polyextraction}
Suppose $f\in \mathcal{J}_{\alpha,M,N}$ for some $\alpha\in \mathbb{R}$ and some
integers $M,N\geq 1$.
Let $r_n=\sum_{k=0}^n\sum_{l=0}^nc_{k,l}z_1^kz_2^l$ be an arbitrary polynomial, let $s_n$ be its projection onto $\mathcal{J}_{\alpha,M,N}$,
\[s_n
=\sum_{\{k\colon Mk,Nk\leq n\}}c_{Mk,Nk}z_1^{Mk}z_2^{Nk},\]
and let $\tilde{s}_n=r_n-s_n$.

Then
\[\|r_nf-1\|_{\alpha}\geq \|s_nf-1\|_{\alpha}.\]
\end{lem}
\begin{proof}
We begin by noting again that monomials of the form $\{z_1^kz_2^l\}$ form
an orthogonal basis for $\mathfrak{D}_{\alpha}$. Next, we have
$s_nf\in \mathcal{J}_{\alpha,M,N}$, and $\tilde{s}_nf \notin \mathcal{J}_{\alpha, M,N}$,
and then, by the previous observation, $s_nf-1\perp \tilde{s}_nf$.

This means that
\begin{align*}
\|r_nf-1\|^2_{\alpha}&=\|s_nf-1+\tilde{s}_nf\|^2_{\alpha}\\
&= \|s_nf-1\|^2_{\alpha}+\|\tilde{s}_nf\|^2_{\alpha}\\
&\geq \|s_nf-1\|^2_{\alpha},
\end{align*}
and the lemma is proved.
\end{proof}
An analogous result holds for functions in the subspace $\mathcal{I}_{\alpha}$.
\begin{proof}[Proof of Theorem \ref{sharpcyclicity}]
We present the details for functions $f\in \mathcal{J}_{\alpha}$; the same
type of arguments work for $\mathcal{J}_{\alpha, M, N}$, with the appropriate inequalities from Lemma
\ref{comparisonlemma} in place of equalities, and also for $f\in \mathcal{I}_{\alpha}$.

We begin by establishing the lower bound.
Let $r_n=\sum_k\sum_lc_{k,l}z_1^kz_2^l$ be any polynomial, and extract the
diagonal part
$s_n$ from $r_n$ as in the preceding lemma.
Note that by construction, $s_nf-1\in \mathcal{J}_{\alpha}$ for each
$\alpha$. By Lemma \ref{polyextraction} and the norm inequality
\eqref{normcomparison}, we obtain
\begin{align*}
\|r_nf-1\|_{\alpha}&\geq \|s_nf-1\|_{\alpha}\\
&=\|R(s_nf-1)\|_{D_{2\alpha}}=
\|R(s_n)R(f)-1\|_{D_{2\alpha}}.
\end{align*}
It is assumed that $R(f)$ satisfies the
hypotheses of Theorem \ref{onevariableoptimal}; the theorem then asserts that
$\dist^2_{D_{2\alpha}}(1,R(f)\cdot \mathfrak{P}_n)\geq \tilde{C}\varphi^{-1}_{\alpha}(n+1)$.
In particular, this yields a lower
bound for $\|R(s_n)R(f)-1\|_{D_{2\alpha}}$, and
the lower bound on $\dist_{\mathfrak{D}_{\alpha}}(1,f\cdot \mathfrak{P}_n)$ follows.

To obtain the upper bound, it is enough to exhibit a concrete sequence of
polynomials $(p_n)$ having $\|p_nf-1\|^2_{\alpha}\leq C(\alpha,f)\varphi^{-1}_{2\alpha}(n+1)$. However, since $R(f)$ satisfies the hypotheses of
Theorem \ref{onevariableoptimal}, there exists a sequence $(q_n)$ of polynomials in
one variable that achieves
\[\|q_n R(f)-1\|_{D_{2\alpha}}^2\leq C(\alpha,f)\varphi^{-1}_{2\alpha}(n+1)\]
for large enough $n$. But then we can define $p_n=L(q_n)\in \mathcal{J}_{\alpha}$, and
the desired estimate follows since
\[\|L(q_n)f-1\|^2_{\alpha}=\|R(L(q_n))R(f)-1\|^2_{D_{2\alpha}}=\|q_n R(f)-1\|^2_{D_{2\alpha}}\]
by Lemma \ref{comparisonlemma}. The proof is complete.
\end{proof}
Note that if $R(f)$ is a polynomial with only simple zeros on the
unit circle $\mathbb{T}$, then it is shown in \cite[Section 3]{BCLSS13} that the
one-variable Riesz polynomials achieve the norm decay obtained above.
In the situation $M=N=1$ then, we have $L(q_n)(z_1,z_2)=p_n(z_1,z_2)$, where
$p_n$ are the Riesz-type polynomials defined in equation \eqref{rieszpolys}.
\section{Polynomials with zeros on
$\partial \mathbb{D}^2$ and measures of finite
energy}\label{examplesSection}
Let us now examine the relationship between cyclicity and boundary zero sets of functions in $\mathfrak{D}_{\alpha}$. Surprisingly,
some functions with large zero sets in some sense \textit{are} cyclic while others
with smaller zero sets are not.
\subsection{Examples}
Let us examine a few simple examples.
\begin{ex}\label{onevariable}
Set $f(z_1,z_2)=1-z_1$. Then $f$ has zero set
\[\mathcal{Z}(f)=\{1\}\times \overline{\mathbb{D}},\]
a (real) $2$-dimensional subset of the topological boundary of $\mathbb{D}^2$
which meets the distinguished boundary along the $1$-dimensional curve
$\{1\}\times \mathbb{T}$.
Note that $f$ is an example of a function of the product type $g(z_1)h(z_2)$ with $g(z_1) = 1-z_1$ and $h(z_2) = 1,$ and therefore
by Proposition \ref{characterize}, $f$ is cyclic in
$\mathfrak{D}_{\alpha}$ if and only if $\alpha \leq 1.$
\end{ex}

\begin{ex}\label{twistedcircle}
Consider the function $f(z_1,z_2)=1-z_1z_2$. The part of the zero set of $f$ that lies on
the boundary of the bidisk,
\[\mathcal{Z}(f)=\{(e^{i\theta}, e^{-i\theta})\colon \theta \in [0, 2\pi)\},\]
can be viewed as a $1$-dimensional real curve contained in the distinguished
boundary $\mathbb{T}^2$. One verifies that all the points in $\mathcal{Z}(f)$ are simple zeros.

Since
\[\frac{1}{f(z_1,z_2)}
=\sum_{k=0}^{\infty}\sum_{l=0}^{\infty}\delta_{k,l}z_1^kz_2^k=\sum_{k=0}^{\infty}z_1^kz_2^k,\]
we have $\|1/f\|^2_{-1}=\sum_{k=0}^{\infty}(1+k)^{-2}<\infty$ but
$\|1/f\|^2_{0}=\sum_{k=0}^{\infty}1=+\infty$, and so $f$ is invertible in
the Bergman space, and indeed in $\mathfrak{D}_{\alpha}$ whenever $\alpha<-1/2$,
but not in the Hardy space of the bidisk.

Nevertheless, by Theorem \ref{sharpcyclicity}, $f$ is cyclic in $\mathfrak{D}_{\alpha}$ if and only if  $\alpha\leq1/2$. Note in particular that this function
is {\it not} cyclic in the classical Dirichlet space of the bidisk!

Explicit computations with the Riesz polynomials in \eqref{rieszpolys} recover the upper bound
in Theorem \ref{sharpcyclicity}.
Namely, we have
\[p_n(z_1,z_2)f(z_1,z_2)-1=-\frac{1}{\varphi_{\alpha}(n+1)}\sum_{k=1}^{n+1}
[\varphi_{\alpha}(k)-\varphi_{\alpha}(k-1)](z_1z_2)^k,\]
and then, since $|\varphi_{\alpha}(k)-\varphi_{\alpha}(k-1)|^2\leq
C(\alpha)(k-1)^{-2\alpha}$, we obtain
\[\|p_nf-1\|^2_{\alpha}\leq \frac{C_1(\alpha)}{(n+1)^{1-2\alpha}}.\]
Thus $\|p_nf-1\|^2_{\alpha}\to 0$ as $n\to \infty$ and
$f$ is cyclic, provided $\alpha\le1/2$.

In fact, considering instead functions of the form $f=1-z_1^Mz_2^N$ for
integer $M,N\geq 1$, and performing the analogous computations, we obtain
\begin{align}
\label{fdecay}
\|p_nf-1\|^2_{\alpha}\leq\frac{C_1(\alpha, M, N)}{(n+1)^{1-2\alpha}}
\end{align}
with a constant $C_{1}(\alpha,M,N)$ which
does not depend on $n$.
\end{ex}

\begin{ex}\label{twocircles}
We examine $f(z_1,z_2)=1-z_1-z_2+z_1z_2=(1-z_1)(1-z_2)$. The zero set of
$f$ is
\[\mathcal{Z}(f)=(\{1\}\times\overline{\mathbb{D}})\cup(\overline{\mathbb{D}}\times\{1\}),\]
a $2$-dimensional set that extends into the
topological boundary of the bidisk. Its intersection with $\mathbb{T}^2$ consists of the curves
\[\mathcal{Z}(f)=(\{1\}\times \mathbb{T})\cup(\mathbb{T}\times \{1\}).\]
All zeros of $f$ are simple,
with the exception of the point $(1,1)$ which has order $2$.
Since
\[\frac{1}{f(z_1,z_2)}=\sum_{k=0}^{\infty}\sum_{l=0}^{\infty}z_1^kz_2^l,\]
it follows that $1/f \notin A^2(\mathbb{D}^2)$.

Note that again, $f$ is separable with
$g(z_1) = 1 - z_1$ and $h(z_2) = 1-z_2$, and therefore $f$ is cyclic in $\mathfrak{D}_{\alpha}$ if and
only if $\alpha \leq 1.$

In this case, computing with the Riesz polynomials leads to misleading estimates. Defining polynomials $p_n$, as before, via \eqref{rieszpolys},
we compute
\begin{align*}
p_nf&=-\frac{1}{(n+1)^{1-\alpha}}\sum_{k=1}^{n+1}[k^{1-\alpha}-(k-1)^{1-\alpha}](z_1^k+z_2^k)\\
&+\frac{1}{(n+1)^{1-\alpha}}\sum_{k=1}^{n+1}[k^{1-\alpha}-(k-1)^{1-\alpha}]z_1^kz_2^k.
\end{align*}
We use the estimates from the previous example, and exploit the one-variable estimates from
\cite{BCLSS13}, to obtain
\begin{align*}
\|p_nf-1\|^2_{\mathfrak{D}_{\alpha}}&=\frac{2}{(n+1)^{2-2\alpha}}\sum_{k=1}^{n+1}(k+1)^{\alpha}[k^{1-\alpha}-(k-1)^{1-\alpha}]^2\\&+\frac{1}{(n+1)^{2-2\alpha}}\sum_{k=1}^{n+1}(k+1)^{2\alpha}[k^{1-\alpha}-(k-1)^{1-\alpha}]^2\\
&\leq\frac{c_1(\alpha)}{(n+1)^{1-\alpha}}+\frac{c_2(\alpha)}{(n+1)^{1-2\alpha}}.
\end{align*}
The first term in the right-hand side dominates when $\alpha<0$, whereas the second is larger when $\alpha>0$. In particular, the estimate does show that $f$ is
cyclic in $\mathfrak{D}_{\alpha}$ provided $\alpha \leq 1/2$. However, as we have seen,
the rate is not optimal, and $f$ remains cyclic when $\alpha>1/2$.
\end{ex}

Note the interesting contrast between Examples \ref{twistedcircle}
and \ref{twocircles}:  the function in Example \ref{twistedcircle}
is not cyclic in the (classical) Dirichlet space of the bidisk, and
yet in some sense has a much smaller zero set than the function in
Example \ref{twocircles}, which is cyclic! On the other hand, as a
kind of dual phenomenon, $f=1-z_1z_2$ exhibits a faster rate of
decay of norms $\|p_nf-1\|_{\alpha}$ for $\alpha<0$ than does
$f=(1-z_1)(1-z_2)$.
\subsection{Measures of finite energy}
It would be interesting to understand the relationship between cyclicity and boundary 
zero sets--in particular, given a function $f$, to find a measure whose support 
lies on the zero set of the boundary values of $f$ that relates to the cyclicity properties of $f$.

We now specialize to the Dirichlet space $\mathfrak{D}=\mathfrak{D}_1$ and give a necessary condition for a function to be cyclic. This condition involves the notion of capacity, and represents a straight-forward generalization of results of Brown and Shields in the one-variable case.

\begin{defn}
Let $E\subset \mathbb{T}^2$ be a Borel set. We say that a probability 
measure $\mu$ supported in $E$ has {\it finite logarithmic energy} if
\begin{equation*}
I[\mu]=\int_{\mathbb{T}^2}\int_{\mathbb{T}^2}\log\frac{e}{|e^{i\theta_1}-e^{i\vartheta_1}|}\log\frac{e}{|e^{i\theta_2}-e^{i\vartheta_2}|}d\mu(\theta_1,\theta_2)d\mu(\vartheta_1,\vartheta_2)<\infty.
\label{logenergy}
\end{equation*}
If $E$ supports no such measure, we say that $E$ has {\it logarithmic capacity} $0$.
\end{defn}
The energy of $\mu$ can be expressed in terms of its Fourier coefficients
\[\hat{\mu}(k,l)=\int_{\mathbb{T}^2}e^{-i(k\theta_1+l\theta_2)}d\mu(\theta_1,\theta_2),
\quad k,l \in \mathbb{Z}.\]
Namely, viewing the integral defining the energy as a convolution with a kernel of positive type (cf. \cite[Chapter 10]{KahBook}), we obtain
\[I[\mu]=\sum_{k=-\infty}^{\infty}\sum_{l=-\infty}^{\infty}\hat{h}(k,l)|\hat{\mu}(k,l)|^2,\]
and computing the Fourier coefficients $\hat{h}(k,l)$ of the product logarithm (see \cite[p. \,294]{BS84} for details), we find that
\begin{equation}
I[\mu]=1+\sum_{k=1}^{\infty}\frac{|\hat{\mu}(k,0)|^2}{k}
+\sum_{l=1}^{\infty}\frac{|\hat{\mu}(0,l)|^2}{l}+\frac{1}{2}\sum_{k \in \mathbb{Z}\setminus \{0\}}\sum_{l=1}^{\infty}\frac{|\hat{\mu}(k,l)|^2}{|k|l}.
\label{logenergycoeffs}
\end{equation}
The notion of energy now allows us to identify some non-cyclic $f\in \mathfrak{D}$ by looking at their
boundary zero sets. To make this notion precise, we note that one can show that functions $f\in \mathfrak{D}$ have radial limits
$f^*(e^{i\theta_1},e^{i\theta_2})=\lim_{r\to 1^{-}}%{r\to 1}
f(re^{i\theta_1},re^{i\theta_2})$ quasi-everywhere. That is, the limit exists for all points outside a set of capacity $0$, and hence it makes sense to speak of the capacity of the set $\mathcal{Z}(f^*)$.
(In fact, Kaptano\u{g}lu considers more general approach regions in
\cite{Kap94}, but we do not need this here.)
\begin{prop}\label{easyBrownShields}
If $f\in \mathfrak{D}$ and $\mathcal{Z}(f^*)$ has positive logarithmic capacity, then $f$ is
not cyclic.
\end{prop}
\begin{proof}
The proof is completely analogous to that of \cite[Theorem 5]{BS84}; we refer the reader to the paper of Brown and Shields for details and present the arguments in condensed form here.

The key idea is to identify the Bergman space $A^2(\mathbb{D}^2)$ with the dual of $\mathfrak{D}$ via the pairing
\[\langle f,g\rangle=\sum_{k=0}^{\infty}\sum_{l=0}^{\infty}a_{k,l}b_{k,l},\]
where $f=\sum_{k,l}a_{k,l}z_1^kz_2^l\in \mathfrak{D}$ and $g=\sum_{k,l}b_{k,l}z_1^kz_2^l\in A^2(\mathbb{D}^2)$.
We then consider the
Cauchy integral $C[\mu]=\int_{\TT^2}(1-e^{i\theta_1}z_1)^{-1}(1-e^{i\theta_2}z_2)^{-1}d\mu(\theta_1,\theta_2)$ of $\mu$, a measure of finite logarithmic energy with $\mathrm{supp}(\mu) \subset E$. A comparison with \eqref{logenergycoeffs} then reveals that
\[\|C[\mu]\|^2_{A^2(\mathbb{D}^2)}=\sum_{k=0}^{\infty}\sum_{l=0}^{\infty}\frac{|\hat{\mu}(k,l)|^2}{(k+1)(l+1)}<\infty\]
so that $C[\mu]$ induces a non-trivial element of $\mathfrak{D}^*$. On the other hand, since the measure
$\mu$ is supported on $\mathcal{Z}(f^*)$ by assumption, the functional induced by $C[\mu]$ annihilates
$[f]$, and so $f$ is not cyclic.
\end{proof}
The argument used in the proof of Proposition \ref{easyBrownShields} 
can be used to give another proof of the 
non-cyclicity of the function $f(z_1,z_2)=1-z_1z_2$ in $\mathfrak{D}$.
Namely, consider the probability measure $\mu_{\mathcal{Z}}$ on
$\mathbb{T}^2$ induced by the (normalized) {\it integration current}
associated with the variety $\mathcal{Z}(1-z_1z_2)\cap \mathbb{T}^2$
(see \cite[Chapter 2]{LeLGruBook} for the relevant definitions). A
quick computation reveals that
$\hat{\mu}_{\mathcal{Z}}(k,l)=\delta_{kl}$, so that
$C[\mu_{\mathcal{Z}}](z_1,z_2)=1/(1-z_1z_2)$, a function in the
Bergman space of the bidisk which satisfies
\[\left\langle z_1^kz_2^lf, C[\mu_{\mathcal{Z}}]\right\rangle=0\]
for all $k,l\geq 0$.

\section{Concluding remarks and open problems}
It appears to be a difficult task to characterize the cyclic elements of $\mathfrak{D}_{\alpha}$ for
$\alpha\leq 1$, and many basic questions remain.
For instance, it is natural to ask whether the {\it Brown-Shields conjecture} is true for functions on the bidisk.
\begin{prob}
Is the condition that $f\in \mathfrak{D}$ is outer and $\mathcal{Z}(f^*)$ has logarithmic capacity $0$ sufficient for $f$ to be cyclic?
\end{prob}
This question remains open for the Dirichlet space of the unit disk, and is widely considered to be a challenging problem.

A first step towards understanding cyclic functions in $\mathfrak{D}_{\alpha}$ might be to solve the following natural problem.
\begin{prob}
Characterize the {\it cyclic polynomials} $f\in \mathfrak{D}_{\alpha}$ for each $\alpha\leq 1$.
\end{prob}
An obvious necessary condition for $f$ to be cyclic is that $\mathcal{Z}(f)\cap \mathbb{D}^2=
\varnothing$, and if $f$ is a polynomial that does not vanish in $\mathbb{D}^2$, then $f$ is
cyclic because both $f$ and $1/f$
extend analytically to a larger polydisk.  But the problem appears to be open
for polynomials with $\mathcal{Z}(f)\cap \partial \mathbb{D}^2\neq \varnothing$: we would at least like to
identify the polynomials whose zero sets have positive capacity. We have proved that polynomials that are products of polynomials in one variable are cyclic, and so the zero sets associated with such functions must all have zero capacity.

As we have seen in our examples, it can happen that a polynomial
with a larger zero set, in the topological sense and in the sense of
measure, is cyclic in $\mathfrak{D}_{\alpha}$ for some $\alpha$,
while a polynomial with a smaller zero set is not. We have also
noted that a polynomial that {\it fails} to be cyclic in
$\mathfrak{D}_{\alpha}$ when $\alpha>1/2$ can be ``more" cyclic in
$\mathfrak{D}_{\alpha}$, for $\alpha<0$, than polynomials that are cyclic in all
$\mathfrak{D}_{\alpha}$. We mean this in the in the sense that
$\dist^2_{\mathfrak{D}_{\alpha}}(1,(1-z_1z_2)\cdot
\mathfrak{P}_n)\asymp C\varphi_{2\alpha}(n+1)$ while
$\dist^2_{\mathfrak{D}_{\alpha}}(1,(1-z_1)(1-z_2)\cdot
\mathfrak{P}_n)\asymp C\varphi_{\alpha}(n+1)$.
It would be interesting to develop a rigorous understanding of this
phenomenon.

\textit{Acknowledgments.}
The authors thank Stefan Richter for interesting conversations and helpful 
suggestions.

\end{document}